\documentclass[12pt]{amsart}

\usepackage{geometry}
\usepackage[all]{xy}
\usepackage{graphicx}
\usepackage{amssymb}
\usepackage{amsfonts}
\usepackage{epstopdf}
\usepackage{amsmath, amsthm}
\usepackage{mathrsfs}
\usepackage{rotating}

\newcommand{\bb}[1]{\mathbb{#1}}
\newcommand{\cc}[1]{\mathcal{#1}}
\newcommand{\lie}[1]{\mathfrak{#1}}

\newcommand{\diag}{\textrm{diag}}
\def\inv{^{-1}}

\theoremstyle{plain}
\newtheorem{theorem}{Theorem}[section]

\newtheorem{corollary}[theorem]{Corollary}
\newtheorem{lemma}[theorem]{Lemma}
\newtheorem{proposition}[theorem]{Proposition}

\theoremstyle{definition}

\newtheorem{remark}[theorem]{Remark}

\newtheorem{definition}[theorem]{Definition}

\begin{document}
\title[Displacing Lagrangians in the manifolds of full flags in $\bb{C}^3$.]
{Displacing (Lagrangians) in the manifolds of full flags.} \author{Milena Pabiniak}
\address{Milena Pabiniak,
Department of Mathematics,  
University of Toronto, 
Toronto ON, Canada}
\email{pabiniak@math.toronto.edu}

\thanks{\today}

\maketitle
\begin{abstract}
In symplectic geometry a question of great importance is whether a (Lagrangian) submanifold is displaceable, that is, if it can be made disjoint from itself by the means of a Hamiltonian isotopy.
In these notes we analyze the coadjoint orbits of $SU(n)$ 
and their Lagrangian submanifolds that are fibers of the Gelfand-Tsetlin map. We use the coadjoint action to displace a large collection of these fibers. Then we concentrate on the case $n=3$ and apply McDuff's method of probes to show that ''most'' of the generic Gelfand-Tsetlin fibers are displaceable. ''Most'' means ''all but one'' in the non-monotone case, and means ''all but a $1$-parameter family'' in the monotone case. In the case of non-monotone manifold of full flags we present explicitly an unique non-displaceable Lagrangian fiber $(S^1)^3$. This fiber was already proved to be non-displaceable in \cite{NNU}. Our contribution is in displacing other fibers and thus proving the uniqueness. 
\end{abstract}
\section{Introduction}
In symplectic geometry one can observe a rigidity of intersections: certain submanifolds are forced to intersect each other in more points than an argument from algebraic or differential topology would predict. This motivates the following question: given a (usually Lagrangian) submanifold $L$ of a symplectic manifold $(M,\omega)$ does there exist a Hamiltonian diffeomorphism $\phi \in \,Ham\,(M, \omega)$ such that $\phi(L) \cap \L=\emptyset$. If such $\phi$ exists, we call $L$ {\bf displaceable}. Otherwise $L$ is called {\bf non-displaceable}.

Hamiltonian torus actions have interesting applications in answering this question. Generic fibers of momentum map are coistoropic submanifolds. In the case of toric actions (i.e. when dimension of torus acting is equal to half of the dimension of manifold) generic fibers are Lagrangians.

An important class of examples of symplectic manifolds is given by coadjoint orbits of Lie groups.   
A Lie group $G$ acts on $\lie{g}^*$, the dual of its Lie algebra, through the coadjoint action. Each orbit 
$M$ of the coadjoint action 
is naturally equipped with the Kostant-Kirillov symplectic form.  For example, when $G=SU(n)$
the group of (complex) unitary matrices, a coadjoint orbit can be identified with the set of traceless Hermitian matrices with
a fixed set of eigenvalues. It can be also viewed as a manifold of flags (full of partial, depending on the orbit) in $\bb{C}^n$ with appropriately scaled symplectic form.
The coadjoint action of the maximal torus of $SU(n)$ is Hamiltonian, but not toric (except $n=2$ case). This action can be extended to a ``Gelfand-Tsetlin'' action which is toric, but is defined only on an open dense subset of the orbit. Generic fibers of the Gelfand-Tsetlin momentum map are Lagrangians and one may ask about their displaceability. Nishinou, Nohara and Ueda proved in \cite{NNU} that at least one of these fibers in non-displaceable. This paper is a step towards answering the question of uniqueness of such fiber. We explicitly displace a large collection of fibers of momentum maps (for standard action and for Gelfand-Tsetlin action). In particular we prove that for non-monotone orbits of $SU(3)$ there is unique non-displaceable Gelfand-Tsetlin fiber.
\begin{theorem}\label{main}
 In the case of non-monotone regular $SU(3)$ coadjoint orbit through $diag\,(a,b,-a-b)\in \lie{su}(3)^*=$ the Gelfand-Tsetlin fiber above the point 
\begin{displaymath}
\begin{cases}
 (\frac a 2, -\frac a 2 ,0)& \textrm{ if }b<0\\
 (\frac{a+b}{ 2}, -\frac{a+b}{ 2} ,0)& \textrm{ if }b>0
\end{cases}
\end{displaymath}
is the unique fiber of the Gelfand-Tsetlin map that is non-displaceable.
\end{theorem}
\textbf{Organization}. Section \ref{Preliminaries} provides background about standard action and the Gelfand-Tsetlin action.
\\
\textbf{Acknowledgments.} The author is very grateful to Yael Karshon for suggesting this problem and helpful conversations during my work on this project. 
The author also would like to thank Leonid Polterovich and Strom Borman for useful discussions.

\section{Preliminaries}\label{Preliminaries}
\subsection{Standard action}\label{standardaction}
Consider the Lie group $G=SU(n)$. We identify the dual of its Lie algebra, $\lie{su}(n)^*$, with the vector space of $n \times n$ traceless Hermitian matrices. 
 Choose the maximal torus of $SU(n)$ to be 
$T_{st}=\{ \diag (e^{it_1},\ldots,e^{it_{n-1}},e^{-i(t_1+\ldots+t_{n-1})})\}$, 
and the positive Weyl chamber to consist of diagonal Hermitian matrices with non-increasing diagonal entries. Let 
$\lambda=\diag (\lambda_1,\ldots, \lambda_n)$, $\lambda_1>\ldots>\lambda_n$, $\sum \lambda_i=0$, be a point in the interior of positive Weyl chamber. The coadjoint action of $SU(n)$ on $\lie{su}(n)^*$ is by conjugation.
Let $M$ be the orbit of coadjoint action of $SU(n)$ on $\lambda$. It is a symplectic manifold with Kostant-Kirillov symplectic form, of dimension $\frac 1 2 n(n-1)$.
The coadjoint action is Hamiltonian with momentum map the inclusion $\Phi:M \hookrightarrow \lie{su}^*(n)$. 
The action of $T_{st}$ (subaction of coadjoint action) is also Hamiltonian  with momentum map 
$\mu:M \rightarrow \lie{t}^*_{st} \cong \bb{R}^{n-1}=\{(x_1,\ldots,x_n)\in \lie{t}^*_{st}\cong \bb{R}^n|\,\,\sum x_i=0\}\subset \lie{t}^*_{st}$ 
sending a matrix in $M$ to its diagonal entries: $\mu([a_{ij}])=(a_{11},\ldots,a_{nn})$. 
Denote by $\cc{Q}=\cc{Q}_{\lambda} \subset \bb{R}^{n-1}$ the polytope that is the image of momentum map $\mu$. Left picture on figure \ref{momentumimages} presents this polytope with additional data, so called ``x-ray`` \footnote{The x-ray of $(M,\omega,\phi)$ is the 
collection of convex polytopes 
 $\phi(X)$ over all connected components $X$ of $M^K$ for some subtorus $K$ of $T$ (for more details see \cite{To}).}, for the case $n=3$.
\begin{figure}\label{momentumimages}
 \includegraphics[width=0.2\textwidth]{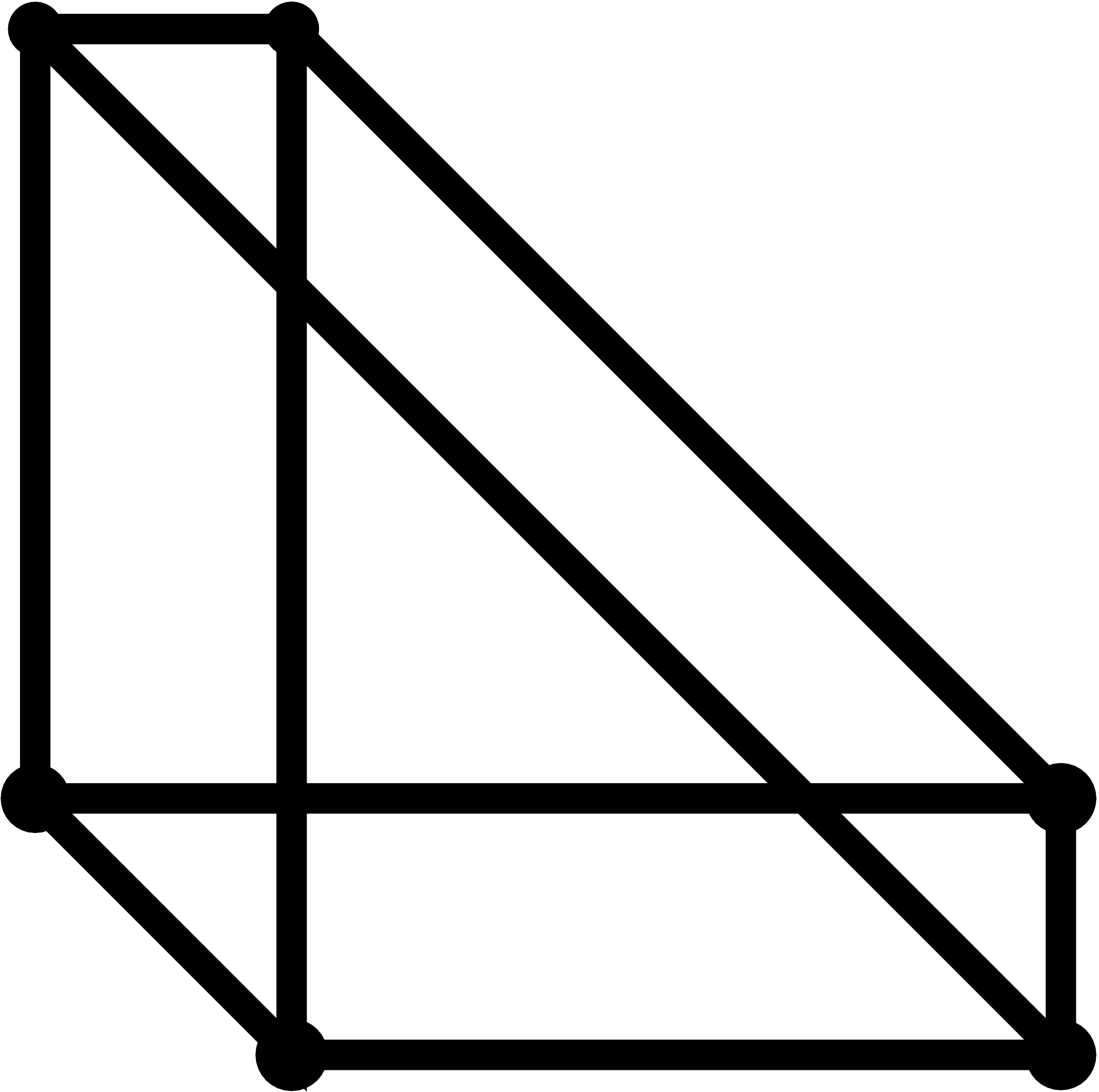} \hspace{1in} \includegraphics[width=0.2\textwidth]{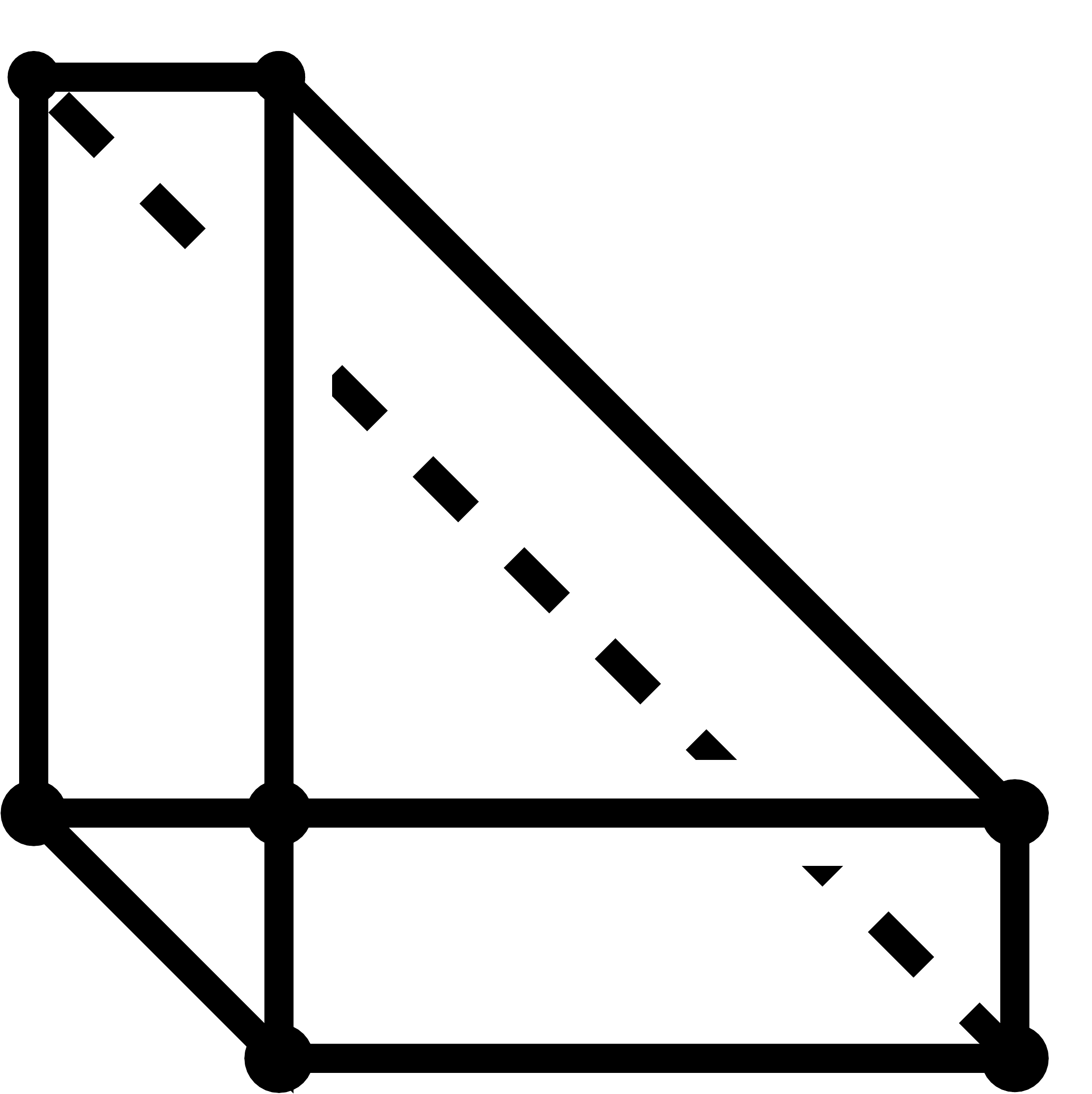}
\caption{The x-ray for the standard $T_{st}$ action and the image of the Gelfand-Tsetlin functions for a regular $SU(3)$ orbit}
\label{momentumimages}
\end{figure}

\subsection{The Gelfand-Tsetlin system}\label{Gelfand-TsetlinSystem}
In this subsection we recall the Gelfand-Tsetlin (sometimes spelled Gelfand-Cetlin, or Gelfand-Zetlin) system of action coordinates, which originally appeared in \cite{GS1}. It is related to the classical Gelfand-Tsetlin polytope introduced in \cite{GTs}.
There are many references describing this system, for example \cite{P4}, \cite{GS1}, \cite{K}. Here we concentrate only on the $SU(3)$ case.

For any matrix $A \in M$, and any $k=1,\ldots, n-1$ let $\lambda^{(k)}_1(A) \geq \ldots \geq \lambda^{(k)}_k(A)$ denote the eigenvalues of $k \times k$ top left minor of $A$.
This defines $\frac 1 2 n(n-1)$ continuous, not everywhere smooth functions from $M$ to $\bb{R}$. The eigenvalues depend smoothly on matrix entries but ordering them may violate the smoothness at points where eigenvalues coincide. The system $\Lambda=\{\lambda^{(k)}_j;\;\;1 \leq k \leq n-1,\,1\leq j \leq k\}\colon M \rightarrow \bb{R}^{\frac 1 2 n(n-1)}$ is called the {\bf Gelfand-Tsetlin} system of action coordinates. Let 
$$U:=\{A \in M;\,\,\forall_{k}\; \lambda^{(k)}_1(A) > \ldots> \lambda^{(k)}_k(A)\} \subset M.$$
The Gelfand-Tsetlin functions are smooth on $U$ and there they integrate to an action of a torus $T_{GT} \cong (S^1)^{\frac 1 2 n(n-1)}$ called the {\bf Gelfand-Tsetlin action} (\cite{GS1},\cite{P4}). This action is Hamiltonian with momentum map the restriction of $\Lambda$ to $U$. The image  $\Lambda(M)$ is called the {\bf Gelfand-Tsetlin polytope} (see right picture in Figure \ref{momentumimages}). We denote it by $\cc{P}=\cc{P}_{\lambda}$. 

The standard action of the maximal torus is a subaction of the Gelfand-Tsetlin action. Therefore there is a projection map 
$pr: \lie{t}_{GT}^*\cong \bb{R}^{\frac 1 2 n(n-1)} \rightarrow   \bb{R}^n\cong  \lie{t}_{st}^* $ given by
$$ pr (\{\lambda^{(j)}_l\})=\Bigl( \lambda^{(1)}_1,\,(\lambda^{(2)}_1+\lambda^{(2)}_2-\lambda^{(1)}_1)\,,
 \ldots ,\, \sum_i \lambda^{(n-1)}_i\, -\sum_i \lambda^{(n-2)}_i, \sum_i \lambda_i\, -\sum_i \lambda^{(n-1)}_i \Bigr),$$
which maps $\cc{P}$ to $\cc{Q}$. Note that $\mu =pr \circ \Lambda$.

\subsection{The Gelfand-Tsetlin polytope}
The classical mini max principle (see for example Chapter I.4 in \cite{CH}) implies that 
\begin{equation*}
 \lambda^{(l+1)}_j(A) \geq  \lambda^{(l)}_j(A) \geq \lambda^{(l+1)}_{j+1}(A).
\end{equation*}
These inequalities, taken over $l=1,\ldots,n-1$, $j=1,\ldots,l$, (with the convention that $\lambda^{(n)}_{j}(A)=\lambda_j$), cut out a polytope in $\bb{R}^{(1/2)\,n(n-1)}.$ In fact this is exactly the Gelfand-Tsetlin polytope $\cc{P}=\Lambda(M)$ (\cite{GS1},\cite{P4}).

\section{Displacing through Hamiltonian isotopies}\label{displacing}
 In this section we prove that certain subsets of the coadjoint orbit are displaceable by means of Hamiltonian isotopies.
\subsection{Displacing fibers of the standard action}
We continue to denote by $M$ a regular $SU(n)$ coadjoint orbit through $(\lambda_1,\ldots,\lambda_n)$, $\sum\,\lambda_i=0$. Recall from Section \ref{standardaction} that the standard action of the maximal torus of $SU(n)$ is Hamiltonian, with momentum map $\mu$.

The fibers of $\mu$ above interior points of $\cc{Q}$ are coistoropic submanifolds of dimension $n(n-2)$, (so not Lagrangian except if $n=2$). 
\begin{proposition}\label{centralnondisplaceable}
 The fiber of $\mu$ above interior point of $\cc{Q}$, $x=(x_1,\ldots,x_n)\neq(0,\ldots,0)$, is displaceable.
\end{proposition}
\begin{proof}
For any even permutation $\sigma \in A_n \subset S_n$ its matrix, $P_{\sigma}$ is in $SU(n)$. Therefore conjugating with $P_{\sigma}$ is a Hamiltonian isotopy ($SU(n)$ is connected). Note that the diagonal entries change in the following way under this conjugation:
$$\mu (P_{\sigma}\,A\,P_{\sigma})=\sigma (\mu(A)).$$
Thus $\mu (P_{\sigma}\,A\,P_{\sigma})=\mu(A)$ for all $\sigma \in A_n$ if and only if $a_{11}=\ldots=a_{nn}$.
Therefore for any $x=(x_1,\ldots,x_n)\neq(0,\ldots,0)$, there exists a $\sigma \in A_n$ such that a Hamiltonian isotopy $P_{\sigma}$ displaces the fiber $\mu \inv (x)$.
\end{proof}


\subsection{Displacing fibers of the Gelfand-Tsetlin system in general case}
In this subsection we will analyze the problem of displacing generic fibers of the Gelfand-Tsetlin system. The fibers above interiors points of $\cc{P}$ are Lagrangian tori $(S^1)^{(1/2)\,n(n-1)}$. The fibers above boundary points which are contained in $U$ are isotropic and of dimension smaller then $\frac 1 2 \dim M$, therefore they are displaceable. The fibers above boundary points not contained in $U$ might be of larger dimension. For example, in $n=3$ case the fiber above the unique not smooth point of $\cc{P}$ (unique $4$-valent vertex) is a Lagrangian sphere $S^3$ (\cite{I}). 

Let $\cc{W}=pr \inv ((0,\ldots,0))$ where $pr$ is the projection $pr: \lie{t}_{GT}^* \rightarrow   \lie{t}_{st}^* $ satisfying $pr \circ \Lambda=\mu$ defined in Section \ref{Gelfand-TsetlinSystem}.
\begin{proposition}\label{stdnondisplace}
 The Gelfand-Tsetlin fiber above any point $x=(x_1, \ldots,x_n)\in \, \cc{P} \setminus \cc{W}$ is displaceable. 
\end{proposition}
\begin{proof}
 Note that this fiber, $\Lambda \inv (x)$, is contained in the fiber $\mu \inv (pr(x))$. Our assumptions guarantee that $pr(x) \neq (0,\ldots,0)$ therefore conjugation with appropriate matrix $P_{\sigma}$ displaces the whole set $\mu \inv (pr (x))$ off itself (see Proposition \ref{centralnondisplaceable}). Thus its subset, $\Lambda \inv (x)$, is displaceable as well.
\end{proof}
\subsection{Displacing fibers of the Gelfand-Tsetlin system in the case $n=3$}
From now on we concentrate on the case $n=3$. First we describe explicitly the Gelfand-Tsetlin polytope for $SU(3)$ coadjoint orbit through a point $\diag (a,b,-a-b)$, $a >b$. 
The Gelfand-Tsetlin polytope $\cc{P}$ 
 (presented on the Figure \ref{polytope}), consists of points $x=(x_1, x_2,x_3)\in \bb{R}^3\cong \lie{t}^*_{GT}$ satisfying the following inequalities (\cite{GS2}, \cite{P4})
$$a\geq x_1 \geq b,$$
$$b\geq x_2 \geq -a-b,$$
$$x_1\geq x_3 \geq x_2.$$

Denote the facets of $\cc{P}$:\begin{align*}
& \textrm{facet}&&&\textrm{ primitive inward normal}\\
 &F_1:&\langle x,e_1\rangle &= a&(-1,0,0)\\
 &F_2:&\langle x,-e_1\rangle &= b& (1,0,0)\\
&F_3:&\langle x,e_2\rangle &= b & (0,-1,0)\\
&F_4:&\langle x,-e_2\rangle &= a+b& (0,1,0)\\
&F_5:&\langle x,e_3-e_1\rangle &= 0& (1,0,-1)\\
&F_6:&\langle x,e_2-e_3\rangle &= 0& (0,-1,1)
\end{align*}

\begin{figure}
 \includegraphics[width=.5\textwidth]{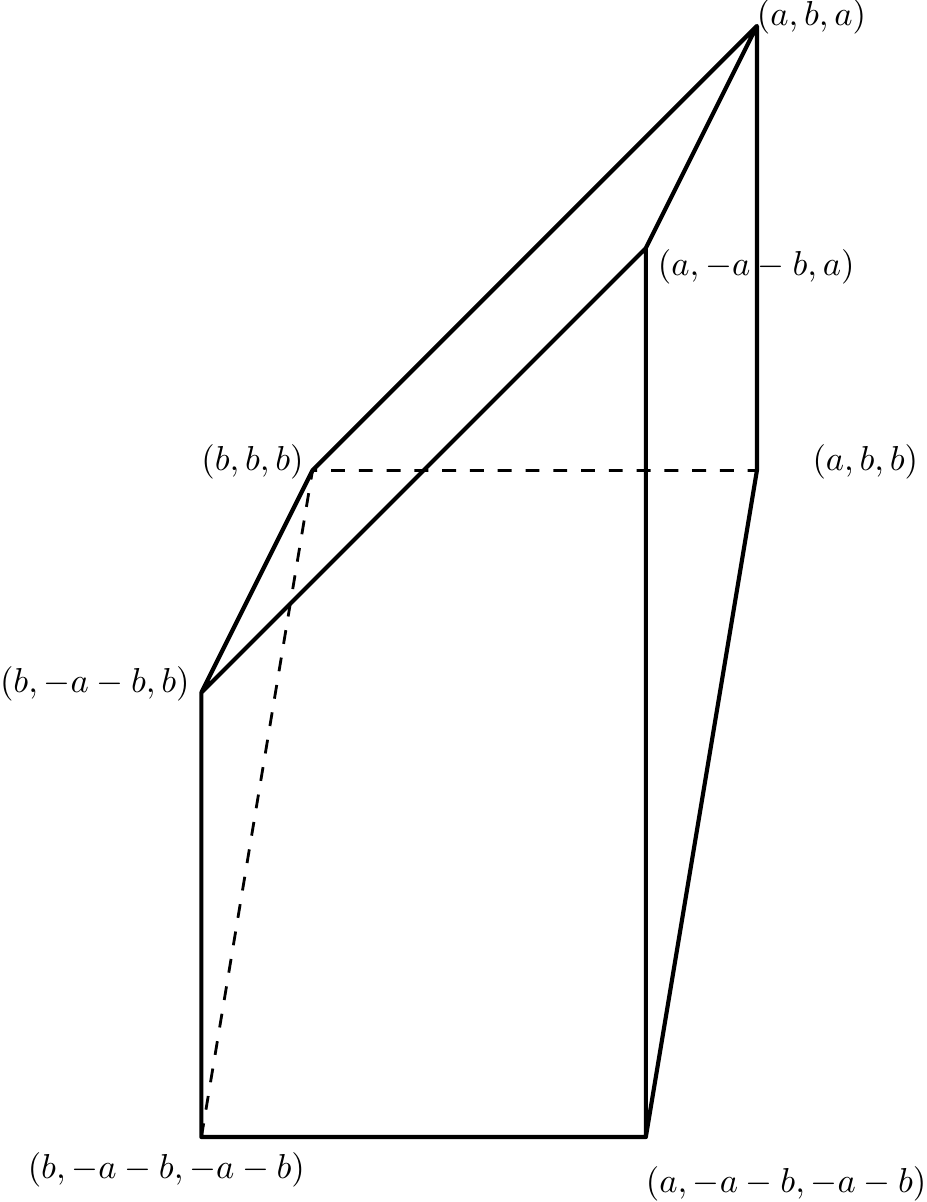}\caption{ The Gelfand-Tsetlin polytope for regular $SU(3)$ orbit through $(a,b,-a-b)$.}
\label{polytope}
\end{figure}

To displace more Gelfand-Tsetlin fibers we use McDuff's method of probes.
Recall the necessary definitions and lemmas.
\begin{definition}\cite{MD}
 Let $w$ be a point of some facet $F$ of a rational polytope $\Delta$ and $\alpha \in \bb{Z}^n$ be integrally transverse to $F$. The {\bf probe } $p_{F,\alpha}(w)=p_{\alpha}(w)$ with the direction $\alpha \in \bb{Z}^n$ and initial point $w \in F$ is the half open line segment consisting of $w$ together with the points in int$\Delta$ that lie on the ray from $w$ in the direction of $\alpha$.
\end{definition}
\begin{lemma}\cite[Lemma 2.4]{MD} \label{probelemma} Let $\Delta$ be a smooth moment polytope. Suppose that a point $u \in \textrm{ int }\Delta$ lies on the probe $p_{F,\alpha}(w)$. Then if $w$ lies in the interior of $F$ and $u$ is less then halfway along $p_{F,\alpha}(w)$, the fiber $L_u$ above $u$ is displaceable.  
\end{lemma}
The lemma is true for any notion of the length along a line. We will use the affine distance.

In the above lemma $\Delta$ is a smooth polytope, but this condition is not really necessary. The only thing we need is the existence of a Darboux chart on $M$ containing the whole preimage of the probe. This is true in our case. The only non-smooth point of the Gelfand-Tsetlin polytope $\cc{P}$ is the vertex $(b,b,b).$ As noted above, the set $U=\Lambda \inv (\cc{P} \setminus \{(b,b,b)\})$ is equipped with the smooth Gelfand-Tsetlin action. Therefore it is a toric (not compact) manifold and we can apply the above lemma for our polytope $\cc{P}$.

\begin{lemma}\label{fromf1}
 The fiber above a point $x \in \textrm{int }\cc{P}$ such that $x_3 \leq b$ and $a>x_1 > \frac{ a+b}{ 2}$ is displaceable by a probe from $F_1$ in the direction $(-1,0,0)$.
\end{lemma}
\begin{proof}Note that $(-1,0,0)$ is integrally transverse to $F_1$.
 Take $x\in \textrm{int }\cc{P}$ such that $x_3 \leq b$ and $a>x_1 > \frac{ a+b}{ 2}$. Then $w=(a,x_2,x_3)$ is in the interior of the facet $F_1$ and the probe from $w$ in the direction $(-1,0,0)$ is the set $\{(a-t,x_2,x_3);\,t \in [0,a-b)\}$.
As $x_1 > a-\frac{a-b}{2}=\frac {a+b}{ 2}$, the point $(x_1,x_2,x_3)$ is displaceable by Lemma \ref{probelemma}.
\end{proof}
Due to the symmetry of $\cc{P}$ we can similarly show:
\begin{lemma}\label{fromf4}
 The fiber above a point $x \in \textrm{int }\cc{P}$ such that $x_3 \geq b$ and $-a-b<x_2 <-\frac{ a}{ 2}$ is displaceable by a probe from $F_4$ in the direction $(0,1,0)$.
\end{lemma}
Note that the fibers above points on the boundary of $\cc{P}$, other then the point $(b,b,b)$, are isotropic tori of dimension less then $3=\frac 1 2 \dim M$, so they are also displaceable. The fiber $\Lambda \inv (b,b,b)$ is a Lagrangian $S^3$ (\cite{I}). 

Therefore from Proposition \ref{stdnondisplace} and Lemmas \ref{fromf1}, \ref{fromf4} we deduce:
\begin{corollary}\label{displacing}
 The only Gelfand-Tsetlin fibers that can possibly be non-displaceable are the fibers above points $(x_1,-x_1,0) \in \cc{P}$ with 
\begin{displaymath}
 \begin{cases}
  b<x_1\leq \frac{a+b}{2}&\textrm{ if }b>0\\
0\leq x_1 \leq \frac a 2 &\textrm{ if }b=0\\
0<x_1 \leq \frac{a}{2}&\textrm{ if }b<0
 \end{cases}
\end{displaymath}
\end{corollary}
\begin{proof}
 According to the above observation only fibers above interior points or above $(b,b,b)$ could be non-displaceable. Note that the set $\cc{W}$ in Proposition \ref{stdnondisplace} is
$$\cc{W}=pr \inv (0)=\{(x_1,-x_1,0) \in \cc{P}\}.$$
If $b \geq 0$ then applying Lemma \ref{fromf1} we displace fibers above points $(x_1,-x_1,0)$ with $a>x_1>\frac{a+b}{2}$. If $b\leq 0$, we apply Lemma \ref{fromf4} and displace fibers above points $(x_1,-x_1,0)$ with $a+b>x_1>\frac{a}{2}$. 
\end{proof}
\begin{remark}
The Lagrangian sphere $\Lambda \inv (b,b,b)$ could be non-displaceable only in the case $b=0$ (the case of monotone orbit).
\end{remark}

\subsection{The monotone case}

A symplectic manifold $(M, \omega)$ is called {\bf spherically monotone} if there exists $k>0$ such that for any class $X$ in the image of Hurewicz homomorphism $\pi_2(M) \rightarrow H_2(M)$ have that $$c_1(TM)[X]=k\, \omega(X).$$
Above $c_1(TM)$ denotes the first Chern class.

In this subsection we consider the case of $b=0$, that is $M$ is the $SU(3)$ orbit through $\lambda=(a,0,-a)$ for some $a>0$.
Then $M$ is spherically monotone.
Recall the theorem Entov and Polterovich.
\begin{theorem}\cite[Theorem 2.1]{EP}\label{ep}
Let $M$ be a closed connected rational and spherically monotone symplectic manifold. 
 Any finite-dimensional Poisson commutative subspace of $C^{\infty}(M)$ has at least one non-displaceable fiber. 
\end{theorem}
(In \cite{EP} it is assumed that $M$ is strongly semi-positive. Spherically monotone manifolds are a special class of strongly semi-positive.)

We will consider unnormalized Gelfand-Tsetlin functions. For any matrix $A=[a_{ij}]  \in M$ let $b_3(A)=a_{11}$, and let $b_1(A),b_2(A)$ be the coefficients of the characteristic polynomial of the $2 \times 2$ top left minor of $A$, i.e. the characteristic polynomial is $t^2+b_1 t+b_2$. These three functions, smooth on $M$, are sometimes called {\bf unnormalized Gelfand-Tsetlin functions}. They have the same level sets as the Gelfand-Tsetlin functions $(a_1,a_2,a_3)$. Note that functions $b_1$ and $b_2$ are $U(2)$ invariant, so they Poisson commute. Proposition 3.2 in \cite{GS1} gives that all three functions: $b_1$, $b_2$ and $b_3$ Poisson commute.
The subspace of $C^{\infty}(M)$ generated by the functions $b_1$, $b_2$ and $b_3$ is finite-dimensional and Poisson commutative. 
Now we start the search of this non-displaceable fiber.

The theorem Entov and Polterovich quoted above proves the existence of non-displaceable fiber of the map $(b_1, b_2,b_3)$. Such a fiber is also a fiber of the Gelfand-Tsetlin system. This proves the following proposition.
\begin{proposition}\label{existence}
 For the $SU(3)$ orbit through $(a,0,-a)$, $a>0$ 
there exists at least one non-displaceable fiber of the Gelfand-Tsetlin system. That is, there exists at least one $p \in \cc{P}$ such that $\Lambda\inv(p)$ is non-displaceable.
\end{proposition}
Corollary \ref{displacing} gives that this non-displaceable fiber must be of the form $\Lambda \inv (x_1,-x_1,0)$ for some $0 \leq x_1 \leq \frac a 2 $.
We cannot show that they are non-displaceable, but we can prove
\begin{proposition}
 The fibers above points $(x_1,-x_1,0) \in \cc{P}$ with $0 \leq x_1 \leq \frac a 2$ are not displaceable by probes.
\end{proposition}
\begin{proof} Let $$N:=\{(x_1,-x_1,0) \in \cc{P};\,\,0 \leq x_1 \leq \frac a 2\,\}.$$
 First notice that these fibers are not displaceable by probes from facets $F_1$, $F_4$ because the distance from any of these facets to $N$ along any probe is at least half of the length of the probe. 

The vectors integrally transverse to $F_2$ are of the form $(1, k, l)$ for $k,l \in \bb{Z}$. The line $\{(x_1,-x_1,0)-t(1, k, l);\,t \in \bb{R}\}$ intersects interior of the facet $F_2$ (at a point $(0,-x_1(k+1),-lx_1)$) if and only if 
$$x_1>0\textrm{ and }0<l<k+1<\frac{a}{x_1}.$$ The affine distance from $(0,-x_1(k+1),-lx_1)$ to $(x_1,-x_1,0)$ in the direction $(1, k, l)$ is $x_1$. Therefore the probe from $(0,-x_1(k+1),-lx_1)$ in the direction of $(1, k, l)$ can be used to displace the fiber above $(x_1,-x_1,0)$ if and only if the length of the probe is greater then $2x_1$, that is, if $$(0,-x_1(k+1),-lx_1)+2x_1(1, k, l)=(2x_1,(k-1)x_1,lx_1)\in \,int\, \cc{P}.$$ In particular this means that $2>l$ and $l>k-1>0$. There are no integers $k,l$ satisfying these conditions, therefore the fibers above points in $N$ cannot be displaced by probes from the facet $F_2$.
Similarly one can show that they also cannot be displaced by probes from the facet $F_3$.

The vectors integrally transverse to the facet $F_5$ are of the form $(1,k,0)$ or $(0,k,-1)$, $k \in \bb{Z}$, but only the second family can give probes intersecting $N$. The line $\{(x_1,-x_1,0)-t(0,k,-1);\,t \in \bb{R}\}$ intersects the hyperplane $\{e_1=e_3\}$ at a point $(x_1,-x_1-kx_1,x_1)$. This intersection point is in the interior of the facet $F_5$ if and only if 
$$x_1>0,\,\,0>-x_1-kx_1>-a,\textrm{ and }x_1>-x_1-kx_1,$$ that is, if 
$-1 <k<-1 + \frac{a}{x_1}.$ The distance from $(x_1,-x_1-kx_1,x_1)$ to $(x_1,-x_1,0)$ in the direction $(0,k,-1)$ is $x_1$. Therefore the probe from $(x_1,-x_1-kx_1,x_1)$ in the direction $(0,k,-1)$ can be used to displace the fiber above $(x_1,-x_1,0)$ if and only if 
$$(x_1,-x_1,0)+2x_1(0,k,-1)=(x_1,kx_1-x_1,-x_1) \in \,int\, \cc{P}. $$
In particular this means that $k< 0$. There are no integers $k$ satisfying $k< 0$ together with the condition $-1 <k$ from above. Therefore the fibers above points in $N$ cannot be displaced by probes from the facet $F_5$. Similar argument proves that they cannot be displaced by probes from the facet $F_6$.
\end{proof}
Patient reader can check that these fibers also cannot be displaced by ''extended probes'' (see \cite{ABM} for definition). 

\subsection{Non-monotone case. Proof of Theorem \ref{main}}

Using Floer theory, without assuming monotonicity, 
Nishinou, Nohara and Ueda proved in \cite{NNU} the following theorem.
\begin{theorem} \cite[Theorem 12.1]{NNU} \label{fromnnu}
For any $\lambda \in \lie{su}(n)^*$ let $\cc{O}_{\lambda}$ stand for the $SU(n)$ coadjoint orbit through $\lambda$. There exists $u \in Int\, \Lambda(\cc{O}_{\lambda})$ in the interior of the Gelfand-Tsetlin polytope for $\cc{O}_{\lambda}$ such that the fiber $\Lambda \inv (u)$ is non-displaceable. 
\end{theorem}

Let $M$ be a regular $SU(3)$ coadjoint orbit that is not monotone, that is an orbit through $\lambda=(a>b>-a-b)$ for some $a,b$ with $b \neq0$. The above theorem proves the existence of non-displaceable Gelfand-Tsetlin fiber. Moreover, computations done in Examples in Section 11 of \cite{NNU} (finding a critical point of the potential funciton), imply that the fiber above $(\frac a 2,- \frac a 2,0)$ if $b<0$, or $(\frac{a+b}{2},-\frac{a+b}{2},0)$ if $b>0$, is non-displaceable.
Below, while proving Theorem \ref{main}, we recover this result, not via potential functions methods of \cite{NNU}, but by displacing other fibers using probes. The advantage of our method is that it also proves uniqueness and displaces the boundary fibers (like the Lagrangian sphere $\Lambda \inv (b,b,b)$).

From Corollary \ref{displacing} we know that the non-displaceable fiber must be of the form $\Lambda \inv (x_1,-x_1,0)$ for some $(x_1,-x_1,0)\in \, int\,\cc{P}$. 
The Figure \ref{slices} shows slices of Gelfand-Tsetlin polytope at $x_3=0$ for the cases of $b<0$ and $b>0$, that is, a polytope in $\bb{R}^2$ satisfying $\max (0,b) \leq x_1 \leq a$ and $-a-b \leq x_2 \leq \min(0,b)$. The black bold line segment corresponds to the possibly non-displaceable fibers.
\begin{figure}
 \includegraphics[width=0.43\textwidth]{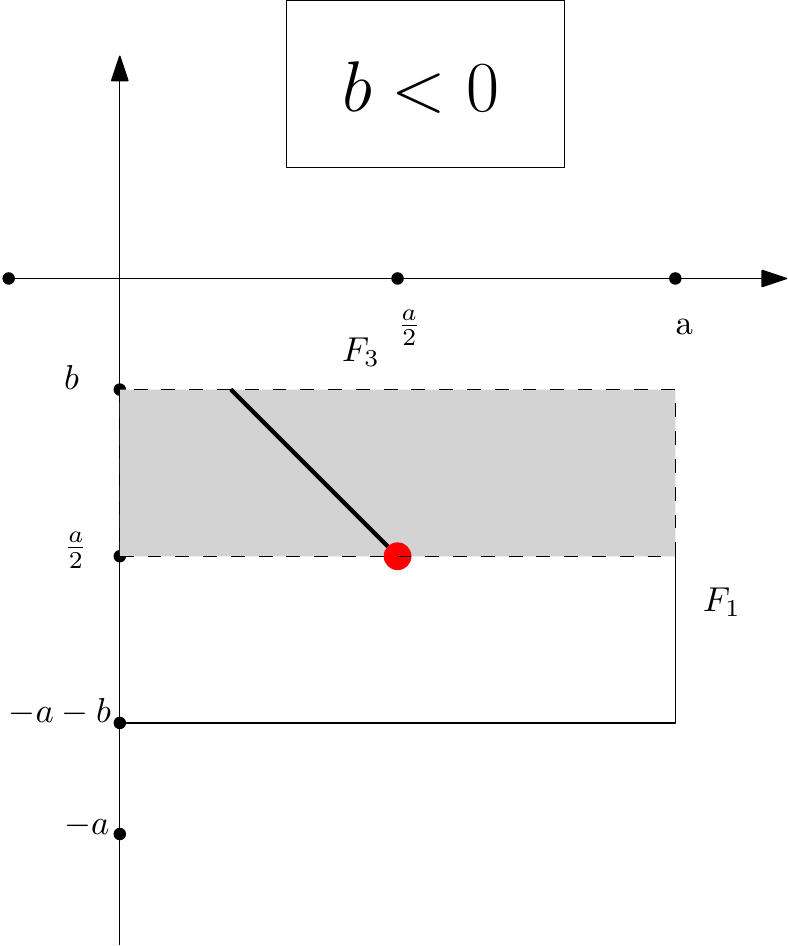} \hspace{0.3in} \includegraphics[width=0.43\textwidth]{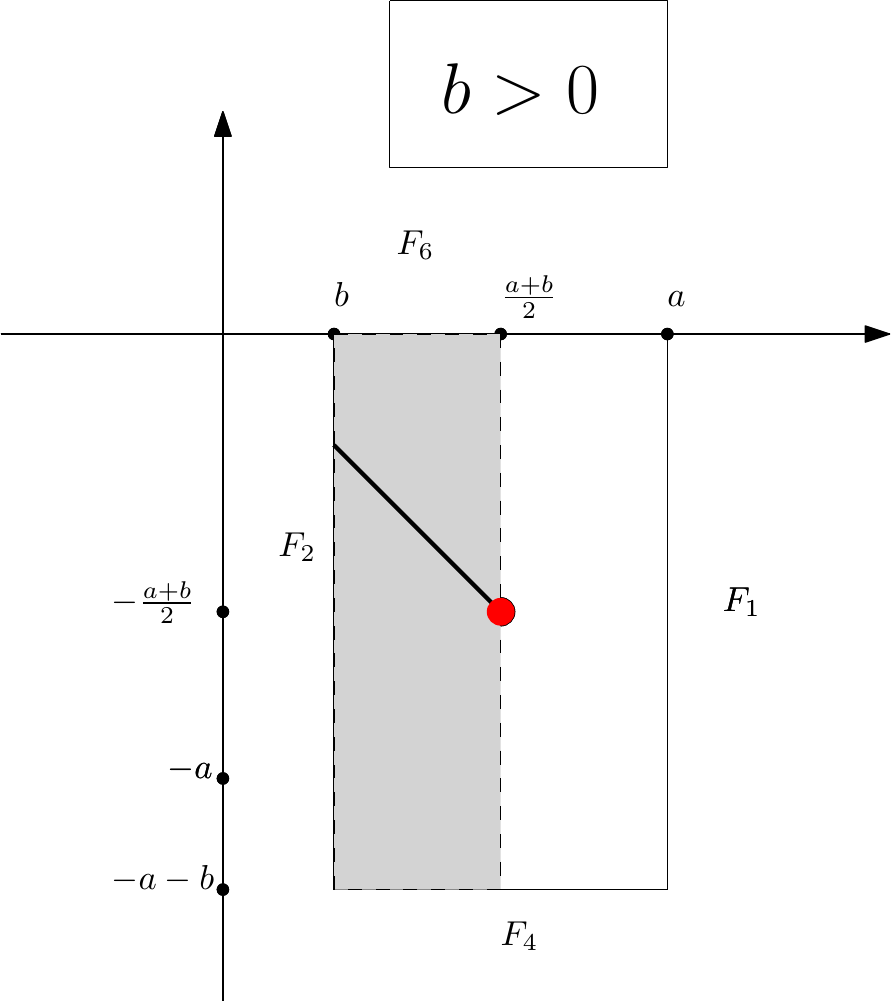}
\caption{The slices of the Gelfand-Tsetlin polytope at $x_3=0$ for $b<0$ and $b>0$.}\label{slices}
\end{figure}
\begin{lemma}\label{fromf3}
 If $b<0$ then the fibers above points $(x_1,x_2,0)$ with $0 < x_1 < a$ and $-\frac a 2 <x_2<b$ (shaded region) are displaceable by probes.
\end{lemma}
\begin{proof}
 The vector $(0,-1,0)$ is integrally transverse to the facet $F_3$. Take any point $(x_1,x_2,0)$ satisfying the above conditions. Then a point $w=(x_1,b, 0)$ is in the interior of $F_3$ and the probe from $w$ in the direction of $(0,-1,0)$ is the set $\{(x_1, b-t, 0);\;t\in[0,a+2b)\}$. Our assumptions imply that $-b<-x_2<\frac a 2$ so 
$x_2= b-t$ with $t< b+\frac{a}{2}$. Therefore the point $(x_1,x_2,0)$ is displaceable by probe from $w$.
\end{proof}
If $b=0$ then the point $w=(x,b,0)=(x,0,0)$ is on the boundary of $F_3$ and the probe could not start from $w$
\begin{lemma}\label{fromf2}
 If $b>0$ then the fibers above points $(x_1,x_2,0)$ with $b < x_1 < \frac{a+b}{2}$ and $-a-b <x_2<0$ (shaded region) are displaceable by probes.
\end{lemma}
\begin{proof}
 The vector $(1,0,0)$ is integrally transverse to the facet $F_2$. Take any point $(x_1,x_2,0)$ satisfying the above conditions. Then a point $v=(b,x_2, 0)$ is in the interior of $F_1$ and the probe from $v$ in the direction of $(1,0,0)$ is the set $\{(b+t, x_2, 0);\;t\in[0,a-b)\}$. Our assumptions imply that $x_1= b+t$ with $t< \frac{a-b}{2}$. Therefore the point $(x_1,x_2,0)$ is displaceable by probe from $v$.
\end{proof}
If $b=0$ then the point $v=(0,x_2,0)$ is on the boundary of $F_1$.

\begin{proof}{\it (of Theorem \ref{main})}
The above Lemmas with Theorem \ref{fromnnu} prove Theorem \ref{main}.
\end{proof}
\bibliographystyle{alpha}
\bibliography{GelfandTsetlinRelated}
\end{document}